\documentclass[12pt, a4paper, reqno]{amsart}
\usepackage{hyperref, fullpage, amsmath, amsfonts, amsthm, amssymb, stmaryrd}

\theoremstyle{plain}
\newtheorem{Theorem}{Theorem}[section]
\newtheorem{Lemma}[Theorem]{Lemma}

\newtheorem{Question}[Theorem]{Question}

\newtheorem{Con}[Theorem]{Construction}

\theoremstyle{definition}

\theoremstyle{remark}

\numberwithin{equation}{section}

\author{Immanuel Halupczok}
\address{School of Mathematics, University of Leeds, Woodhouse Lane, Leeds
LS2 9JT, United Kingdom}
\email{math@karimmi.de}

\author{Franziska Jahnke}
\address{Institut f\"ur Mathematische Logik und Grundlagenforschung,
University of M\"unster, Einsteinstr.\;62, 48149 M\"unster, Germany}
\email{franziska.jahnke@wwu.de} 

\begin{document}

\title{A definable henselian valuation with high quantifier complexity}
\begin{abstract}
We give an example of a parameter-free definable henselian valuation ring which is
neither definable by a parameter-free $\forall\exists$-formula nor by 
a parameter-free $\exists\forall$-formula in the language of rings. 
This answers a question of Prestel. 
\end{abstract}
\maketitle

\section{Introduction}
There have been several recent results concerning definitions of henselian
valuation rings in the language of rings, mostly 
using formulae of low quantifier
complexity (see \cite{CDLM}, \cite{Hon14}, 
\cite{AK14}, \cite{Fe14}, \cite{JK14a}, \cite{Pr14}, \cite{FP14} and 
\cite{FJ14}). 
After a number of these results had been proven,
Prestel showed a Beth-like Characterization Theorem which 
gives criteria for the existence of low-quantifier definitions for
henselian valuations:
\begin{Theorem}[{\cite[Characterization Theorem]{Pr14}}] \label{P}
Let $\Sigma$ be a first order axiom system in the ring language
$\mathcal{L}_{\rm ring}$ together with a unary predicate $\mathcal{O}$. Then there exists an
$\mathcal{L}_{\rm ring}$-formula $\phi(x)$, defining uniformly
in every model $(K, \mathcal{O})$ of $\Sigma$ the set $\mathcal{O}$, 
of quantifier type
\begin{align*}
\exists & \textrm{ iff }(K_1 \subseteq K_2 \Rightarrow \mathcal{O}_1
\subseteq \mathcal{O}_2) \\
\forall & \textrm{ iff }(K_1 \subseteq K_2 \Rightarrow \mathcal{O}_2 \cap K_1
\subseteq \mathcal{O}_1) \\
\exists\forall & \textrm{ iff }(K_1 \prec_\exists K_2 \Rightarrow \mathcal{O}_1
\subseteq \mathcal{O}_2) \\
\forall\exists & \textrm{ iff }(K_1 \prec_\exists K_2 \Rightarrow \mathcal{O}_2
\cap K_1 \subseteq \mathcal{O}_1) 
\end{align*}
for all models $(K_1, \mathcal{O}_1)$, $(K_2, \mathcal{O}_2)$ of $\Sigma$.
Here $K_1 \prec_\exists K_2$ means that $K_1$ is existentially closed in
$K_2$, i.e., every existential 
$\mathcal{L}_{\rm ring}$-formula $\rho(x_1, \dotsc,x_m)$
with parameters from $K_1$ that holds in $K_2$ also holds in $K_1$.
\end{Theorem}

Applying the conditions in Theorem \ref{P}, it is easy to see that most
known parameter-free definitions of henselian valuation rings
in $\mathcal{L}_{\rm ring}$ are in fact equivalent to 
$\emptyset$-$\forall\exists$-formulae or $\emptyset$-$\exists\forall$-formulae.
Consequently, Prestel asked the following:  
\begin{Question}
Let $(K,w)$ be a henselian valued field such that $\mathcal{O}_w$
is a $\emptyset$-definable subset of $K$ in the language $\mathcal{L}_{\rm ring}$.
Is there already a $\emptyset$-$\forall\exists$-formula or a 
$\emptyset$-$\exists\forall$-formula which defines $\mathcal{O}_w$ in $K$?
\end{Question}

The aim of this note is to provide a counterexample to Prestel's question.
More precisely, we show:
\begin{Theorem} There are ordered abelian groups $\Gamma_1$ and $\Gamma_2$ 
such that for any PAC field $k$ with $k\neq k^\mathrm{sep}$ the henselian valuation 
ring \label{main}
$\mathcal{O}_w = 
k((\Gamma_1))[[\Gamma_2]]$ is $\emptyset$-definable in the field
$K=k((\Gamma_1))((\Gamma_2))$. 
However, $\mathcal{O}_w$ is neither definable by a 
$\emptyset$-$\forall\exists$-formula nor by a 
$\emptyset$-$\exists\forall$-formula in $K$.
\end{Theorem}
Moreover, we consider a specific example, namely the case $k=\mathbb{Q}^{\mathrm{tot}\mathbb{R}}(\sqrt{-1})$.
Here, $\mathbb{Q}^{\mathrm{tot}\mathbb{R}}$ denotes the totally real numbers,
that is the maximal extension of $\mathbb{Q}$ such that for every embedding of the field
into the complex numbers the image lies inside the real numbers.
By \cite[Example 5.10.7]{Jar}, the field $\mathbb{Q}^{\mathrm{tot}\mathbb{R}}(\sqrt{-1})$ is an example of a
PAC field.
From the results contained in this paper, it is easy to obtain an 
explicit $\mathcal{L}_\mathrm{ring}$-formula which defines 
$\mathcal{O}_w$ in the
field $$K=\mathbb{Q}^{\mathrm{tot}\mathbb{R}}(\sqrt{-1})((\Gamma_1))((\Gamma_2))$$ 
and which -- by Theorem \ref{main} --
is not equivalent to a $\emptyset$-$\forall\exists$-formula or a 
$\emptyset$-$\exists\forall$-formula modulo $\mathrm{Th}(K)$.

Note that in all examples constructed, $w$ admits proper henselian 
refinements and hence is \emph{not} the canonical henselian valuation of $K$. 
Thus, our results do not contradict Theorem 1.1 in \cite{FJ14} 
which states that the canonical henselian valuation is in most cases
$\emptyset$-$\forall\exists$-definable or $\emptyset$-$\exists\forall$-definable
as soon as it is $\emptyset$-definable at all
(see also \cite{FJ14} for the definition of the canonical 
henselian valuation of a field).

\section{The construction}
\subsection{The value group} \label{groups}
In this section, we consider examples of (Hahn) sums of ordered abelian groups.
For $H$ and $G$ ordered abelian groups, consider the lexicographic sum
$G \oplus H$, that is the ordered group with underlying set $G \times H$ 
and equipped with the lexicographic order such that $G$ is more significant.
More generally, recall that for a totally ordered set $(I,<)$ and a 
family $(G_i)_{i \in I}$ of ordered abelian groups, there is a corresponding
Hahn sum $$G:=\bigoplus_{i \in I} G_i.$$ 
consisting of all sequences $(g_i)_{i \in I} \in \prod_{i \in I} G_i$
with finite support. Componentwise addition and
the lexicographic order (where $G_i$ is more significant than $G_{i'}$ if
$i < i'$) give 
$G$ the structure of an ordered abelian group.
For any $k \in I$, the final segment $\bigoplus_{i \in I,\,i>k} G_i$
 is a convex subgroup of  $G$
and the quotient of $G$ by said subgroup is isomorphic to
the corresponding initial segment $\bigoplus_{i \in I,\,i\leq k} G_i$.

We consider the ordered abelian groups 
$$X:=\mathbb{Z}_{(2)}
=\left\{\dfrac{a}{b} \in \mathbb{Q}\,|\, a,b \in \mathbb{Z},\,2 \nmid b\right\} 
\textrm{ and } 
Y:=\mathbb{Z}_{(3)} =\left\{\dfrac{a}{b} \in \mathbb{Q}\,|\, a,b \in \mathbb{Z},\,3 
\nmid b\right\}$$
as building blocks in the construction of Hahn sums.
All ordered abelian groups considered in this note are of the form
$\bigoplus_{j \in J}G_j$ for some ordered index set $J$ with $G_j \in \{X,Y\}$
for all $j \in J$.
Let $(\mathbb{N},<)$ denote the natural numbers with their
usual ordering and $(\mathbb{N}',<)$ the natural numbers in reverse order.
Define
$$\Gamma_1:= \bigoplus_\mathbb{N} ((\bigoplus_\mathbb{N} Y)\oplus X)$$
and 
$$\Gamma_2:= \bigoplus_{\mathbb{N}'} (X \oplus Y).$$
Then, the ordered abelian group $Y$ is the quotient of
$\Gamma_1$ by its convex subgroup
$$\Lambda_1:=((\bigoplus_{\mathbb{N}\setminus\{0\}} Y)\oplus X))
\oplus (\bigoplus_{\mathbb{N}\setminus\{0\}} ((\bigoplus_\mathbb{N} Y)\oplus X))
.$$
Note that there is an isomorphism $f_1: \Lambda_1 \xrightarrow{\sim} \Gamma_1$
of ordered abelian groups induced by the (unique) isomorphism of the 
index sets. 
Furthermore, $X \oplus Y$ is a
convex subgroup of $\Gamma_2$, with corresponding quotient
$$\Lambda_2=\bigoplus_{\mathbb{N}'\setminus\{0\}} (X \oplus Y).$$
Again, the (unique) isomorphism of the index sets
induces an
isomorphism $g_2:\Lambda_2 \xrightarrow{\sim} \Gamma_2$.
We now consider the lexicographic sum 
$$\Gamma := \Gamma_2 \oplus \Gamma_1.$$
\begin{Lemma} \label{wichtig}
Let $\Gamma$ be as above. Then, the convex subgroup $\Gamma_1$
is a parameter-free 
$\mathcal{L}_\mathrm{oag}$-definable subgroup of $\Gamma$.
\end{Lemma}

\begin{proof}
We write $\Gamma$ as a Hahn sum
\[
\Gamma = \bigoplus_{j \in J} G_j
\]
with $G_j \in \{X, Y\}$.
There is a smallest element $k \in  J$ which has a successor $k'$
such that $G_{k} = G_{k'} = Y$. For that $k$, one has
\[
\Gamma_1 = \bigoplus_{j \in J,\, j > k} G_j;
\]
the idea of this proof is to express this as a formula, using that $J$ is interpretable in $\Gamma$;
see e.g.\ \cite{CH11} or \cite{Schmitt82}. We now explain this
interpretation in some detail.

Fix $r \in \mathbb N$ (we will only consider $r = 3, 6$).
For $x \in \Gamma \setminus r\Gamma$, let $F_r(x)$ be the largest convex subgroup of $\Gamma$
which is disjoint from $x + r\Gamma$. For fixed $r$, $F_r(x)$ is definable uniformly in $x$ by
\cite[Lemma 2.11]{Schmitt82} or \cite[Lemma 2.1]{CH11}, namely:
\[
y \in F_r(x) \iff [0, r\max\{-y, y\}] \cap x + r\Gamma = \emptyset.
\]
Using that all $G_j$ are archimedean, one can check that the set of groups of the form $F_r(x)$
($x \in \Gamma \setminus r\Gamma$) is exactly equal to the set of groups of the form
\[
\bigoplus_{j \in J, j > j_0} G_j,
\]
where $j_0$ runs over those $j \in J$ for which $G_j$ is not $r$-divisible; see
\cite[Example 2.3]{Schmitt82} or a combination of the examples in 
\cite[Sections 4.1 and 4.2]{CH11} for details.

Thus we have the interpretation $J = (\Gamma \setminus 6\Gamma) / \mathord{\sim}_6$, where $x \sim_r x'$ iff
$F_r(x) = F_r(x')$, and
\[
J_Y := (\Gamma \setminus 3\Gamma) / \mathord{\sim}_3 = \{j \in J \mid G_j = Y\}.
\]
Now our $k$ from above is a $\emptyset$-definable element of $J$ and we have
$F_6(x) = \Gamma_1$ for any $x \in \Gamma \setminus 6\Gamma$ with $x/\mathord{\sim}_6 = k$, as desired.
\end{proof}

Next, we give different existentially closed embeddings of $\Gamma$ into itself
which we will use to apply Prestel's Theorem.
We use the following facts:

\begin{Theorem}[{\cite[Corollaries 1.4 and 1.7]{We90}}] \label{We}
Let $G_1$ and $G_2$ be ordered abelian groups.
\begin{enumerate}
\item If $G_1$ is
a convex subgroup of $G_2$, then $G_1$ is existentially closed in $G_2$.
\item Consider the Hahn sum $G=G_2 \oplus G_1$. Let $G_1'$ (resp.\ $G_2'$) 
be an ordered subgroup of $G_1$ (resp.\ $G_2$)
that is existentially closed in 
$G_1$ (resp.\ $G_2$), and put $G':=G_2' \oplus G_1'$. Then $G'$ is
existentially closed in $G$.
\end{enumerate}
\end{Theorem}

The first embedding $f_3:\Gamma \to \Gamma$ 
which we want to consider is given by
$f_1:\Lambda_1 \to \Gamma_1$ (defined above) 
and $f_2:\Gamma_2 \oplus Y \to \Gamma_2$ which maps $\Gamma_2$ isomorphically
to $\Lambda_2$ via $g_2^{-1}$ (defined above) and which embeds $Y$ into
$X \oplus Y$ as a convex subgroup:

\begin{align} \label{f_3}
f_3:\Gamma_2 \oplus \Gamma_1 &= \Gamma_2 \oplus Y & &\oplus \Lambda_1 
\phantom{xxxxxxxxxxxxxxxxxxxx}
 \notag\\
&\cong \underbrace{f_2(\Gamma_2 \oplus Y)}_{\prec_\exists \, \Gamma_2} 
& & 
\oplus \underbrace{f_1(\Lambda_1)}_{= \,\Gamma_1} \notag\\
&\prec_\exists \Gamma_2 & &\oplus \Gamma_1
\end{align}

The second embedding is $g_3: \Gamma \to \Gamma$ given by 
$g_2: \Lambda_2 \to \Gamma_2$
(defined above) and $g_1:(X \oplus Y) \oplus \Gamma_1 \to 
\Gamma_1$ which embeds it as a convex subgroup. More precisely,
we consider the isomorphism 
$$g_{1,1}: \Gamma_1 \xrightarrow{\sim} ((\bigoplus_{\mathbb{N}\setminus\{0\}} Y) 
\oplus X)\oplus \bigoplus_{\mathbb{N}\setminus\{0,1\}} (\bigoplus_{\mathbb{N}}
Y) \oplus X$$
induced by the (unique) order isomorphism of the index sets,
and the embedding
$$g_{1,2}: X \oplus Y \to (\bigoplus_{\mathbb{N}} Y) \oplus X \oplus Y$$
as a convex subgroup which maps $X \oplus Y$ onto itself as a final segment
of the Hahn sum on the right.
Overall, we obtain the following
embedding of $\Gamma$ into itself:
\begin{align}\label{g_3}
g_3: \Gamma_2 \oplus \Gamma_1 & = \Lambda_2 & &\oplus (X \oplus Y) \oplus \Gamma_1
\phantom{xxxxxxxxxxxxxxxx}
\notag\\
&\cong \underbrace{g_2(\Lambda_2)}_{= \Gamma_2} 
& & 
\oplus \underbrace{g_1((X \oplus Y) \oplus \Gamma_1)}_{\prec_\exists \Gamma_1}
\notag\\
&\prec_\exists \Gamma_2 & &\oplus \Gamma_1
\end{align}

\subsection{The residue field}
Let $k$ be a PAC field which is not separably closed. Then,
any henselian valuation with residue field $k$ is
$\emptyset$-definable (\cite[Lemma 3.5 and Theorem 3.6]{JK14a}). 
Moreover, assume that $k$ is a PAC field of 
characteristic $0$ such that the algebraic part 
$k_0$
of $k$ is not algebraically closed, i.e., 
$k_0:= \mathbb{Q}^{alg} \cap k \subsetneq \mathbb{Q}^{alg}$.
By \cite[Theorem 3.5 and its proof]{Fe14}, any henselian valuation with 
residue field $k$ is
$\emptyset$-$\exists$-definable: In fact, for any monic and
irreducible $f \in k_0[X]$ with $\mathrm{deg}(f)>1$, \cite[Section 3]{Fe14}
gives a parameter-free $\mathcal{L}_{\rm ring}$-formula depending on $f$ 
which defines the valuation ring of $v$ in any henselian valued field $(K,v)$ with 
residue field $k$.

In order to get an explicit example, we
consider the maximal totally real extension 
$\mathbb{Q}^{\mathrm{tot}\mathbb{R}}$ of $\mathbb{Q}$. 
As mentioned in the introduction, 
$k:=\mathbb{Q}^{\mathrm{tot}\mathbb{R}}(\sqrt{-1})$ 
is a PAC field by \cite[Example 5.10.7]{Jar}.
Furthermore, as $\sqrt[3]{2}$ is not totally real, $f=X^3-2$
is a monic and irreducible polynomial with coefficients in the algebraic part
$k_0$ of $k$.
Thus, by \cite[Proposition 3.3]{Fe14}, the formula
\begin{align*}
\eta(x)\equiv (\exists u,t)(&x=u+t\, \wedge \,
(\exists y,z,y_1,z_1)(u=y_1-z_1\, \wedge\, y_1(y^3-2)=1\, \wedge \,
z_1(z^3-2)=1) \\
&\wedge
(\exists y,z,y_1,z_1)(t=0\, \vee\, (t=y_1z_1 
\,\wedge\, y_1(y^3-2)=1 \,\wedge\, z_1(z^3-2)=1))
\end{align*}
defines the valuation ring of $v$ in any henselian valued field $(K,v)$ with 
residue field $k$.

\subsection{Power series fields}
Now, define $K:=k((\Gamma_1))((\Gamma_2))=k((\Gamma_2 \oplus \Gamma_1))$ for $k$ PAC but 
not separably closed.
Then, the valuation ring of the henselian valuation $v$ on $K$ with
value group $\Gamma_2 \oplus \Gamma_1$ and residue field $k$ is
$\emptyset$-definable by the results discussed in the previous section.
Moreover, for $k=\mathbb{Q}^{\mathrm{tot}\mathbb R}$, $\mathcal O_v$ is $\emptyset$-$\exists$-definable by
the formula $\eta(x)$ (as above).
Let $w$ be the coarsening of $v$ with value group $\Gamma_2$ and
residue field $k((\Gamma_1))$. Recall that by Lemma \ref{wichtig}, the 
convex subgroup $\Gamma_1$ is $\emptyset$-definable in the ordered
abelian group 
$\Gamma_2 \oplus \Gamma_1$. Thus, $w$ is $\emptyset$-definable on $K$.

We now give two different existentially closed embeddings of $K$ into itself
which combined with Prestel's Characterization Theorem show that $w$
is neither $\emptyset$-$\forall\exists$-definable nor
$\emptyset$-$\exists\forall$-definable.
\begin{Theorem}[Ax-Kochen/Ersov, see {\cite[p.\,183]{KP84}}] \label{AKE}
Let $(K,w)$ be a henselian valued field of equicharacteristic $0$. Let $(K,w) \subseteq (L,u)$
be an extension of valued fields. If the residue field of $(K,w)$ 
is existentially closed in the residue field of $(L,u)$ and
the value group of $(K,w)$ is existentially closed in the value group of 
$(L,u)$, then $(K,w)$ is existentially closed in $(L,u)$. 
\end{Theorem}

\begin{Con} Let $K=k((\Gamma_1))((\Gamma_2))$ with $\Gamma_1$ and $\Gamma_2$
as before. Let $w$ denote the power series valuation on $K$ with
valuation ring $k((\Gamma_1))[[\Gamma_2]]$ and value group $\Gamma_2$. 
\begin{enumerate}
\item Consider the existential embeddings $f_0=\mathrm{id}_k$,
as well as $f_3$ as defined in Equation (\ref{f_3}). By Theorem 
\ref{AKE}, there is an existential embedding $f: K \to K$ which
prolongs $f_0$ and $f_3$. Then, as the embedding maps more than 
just $\Gamma_2$ into $\Gamma_2$, we have $f(\mathcal{O}_w) \supsetneq
\mathcal{O}_w$.
\item On the other hand, consider the existential embeddings 
$g_0=\mathrm{id}_k$,
as well as $g_3$ as defined in Equation (\ref{g_3}). 
Once again, 
there is an existential embedding $g: K \to K$ which
prolongs $g_0$ and $g_3$. Then, as the embedding maps more than 
just $\Gamma_1$ into $\Gamma_1$, we have $g(\mathcal{O}_w) \subsetneq
\mathcal{O}_w$.
\end{enumerate}
\end{Con}

In particular, 
the henselian valuation $w$ with value group $\Gamma_2$ is 
$\emptyset$-definable on 
$$K=\mathbb{Q}^{\mathrm{tot}\mathbb{R}}(\sqrt{-1})((\Gamma_1))((\Gamma_2))$$ 
but neither $\emptyset$-$\forall\exists$-definable nor
$\emptyset$-$\exists\forall$-definable by Theorem \ref{P}.
This finishes the proof of Theorem \ref{main}. 
 
\bibliographystyle{alpha}
\bibliography{franzi}

\begin{thebibliography}{CDLM13}

\bibitem[AK14]{AK14}
Will Anscombe and Jochen Koenigsmann.
\newblock An existential $\emptyset$-definition of ${F}_q[[t]]$ in
  ${F}_q((t))$.
\newblock {\em J. Symbolic Logic}, 79(4):1336--1343, 2014.

\bibitem[CDLM13]{CDLM}
Raf Cluckers, Jamshid Derakhshan, Eva Leenknegt, and Angus Macintyre.
\newblock Uniformly defining valuation rings in {H}enselian valued fields with
  finite or pseudo-finite residue fields.
\newblock {\em Annals of Pure and Applied Logic}, 164(12):1236--1246, 2013.

\bibitem[CH11]{CH11}
Raf Cluckers and Immanuel Halupczok.
\newblock Quantifier elimination in ordered abelian groups.
\newblock {\em Confluentes Math.}, 3(4):587--615, 2011.

\bibitem[Feh15]{Fe14}
Arno Fehm.
\newblock Existential {$\emptyset$}-definability of {H}enselian valuation
  rings.
\newblock {\em J. Symb. Log.}, 80(1):301--307, 2015.

\bibitem[FJ15]{FJ14}
Arno Fehm and Franziska Jahnke.
\newblock On the quantifier complexity of definable canonical henselian
  valuations.
\newblock Preprint, to appear in Mathematical Logic Quarterly, 2015.

\bibitem[FP15]{FP14}
Arno Fehm and Alexander Prestel.
\newblock Uniform definability of henselian valuation rings in the {M}acintyre
  language.
\newblock Preprint, available on arXiv:1408.4816 [math.AC], 2015.

\bibitem[Hon14]{Hon14}
Jizhan Hong.
\newblock Definable non-divisible {H}enselian valuations.
\newblock {\em Bull. Lond. Math. Soc.}, 46(1):14--18, 2014.

\bibitem[Jar11]{Jar}
Moshe Jarden.
\newblock {\em {Algebraic Patching}}.
\newblock Springer Monographs in Mathematics. Springer-Verlag, Berlin, 2011.

\bibitem[JK15]{JK14a}
Franziska Jahnke and Jochen Koenigsmann.
\newblock Definable {H}enselian valuations.
\newblock {\em J. Symb. Log.}, 80(1):85--99, 2015.

\bibitem[KP84]{KP84}
Franz-Viktor Kuhlmann and Alexander Prestel.
\newblock On places of algebraic function fields.
\newblock {\em J. Reine Angew. Math.}, 353:181--195, 1984.

\bibitem[Pre15]{Pr14}
Alexander Prestel.
\newblock Definable henselian valuation rings.
\newblock Preprint, available on ArXiv:1401.4813 [math.AC], to appear in
  Journal of Symbolic Logic, 2015.

\bibitem[Sch82]{Schmitt82}
Peter~H. Schmitt.
\newblock Model theory of ordered abelian groups.
\newblock Habilitationsschrift, Universit\"at Heidelberg, 1982.

\bibitem[Wei90]{We90}
V.~Weispfenning.
\newblock Existential equivalence of ordered abelian groups with parameters.
\newblock {\em Arch. Math. Logic}, 29(4):237--248, 1990.

\end{thebibliography}

\end{document}